\documentclass[english]{amsart}

\usepackage{amsmath,amssymb,enumerate,bbm}

\usepackage[T1]{fontenc}
\usepackage[utf8]{inputenc}
\usepackage[all]{xy}

\usepackage{babel}
\usepackage{amstext}
\usepackage{amsmath}
\usepackage{amsfonts}
\usepackage{latexsym}
\usepackage{ifthen}

\usepackage{xypic}
\xyoption{all}
\newcommand{\cal}{\mathcal}

\newtheorem{lemma1}[equation]{}

\newenvironment{lemma}{\begin{lemma1}{\bf Lemma.}}{\end{lemma1}}
\newenvironment{example}{\begin{lemma1}{\bf Example.}\rm}{\end{lemma1}}

\newenvironment{theorem}{\begin{lemma1}{\bf Theorem.}}{\end{lemma1}}
\newenvironment{proposition}{\begin{lemma1}{\bf Proposition.}}{\end{lemma1}}
\newenvironment{corollary}{\begin{lemma1}{\bf Corollary.}}{\end{lemma1}}
\newenvironment{remark}{\begin{lemma1}{\bf Remark.}\rm}{\end{lemma1}}
\newenvironment{definition}{\begin{lemma1}{\bf Definition.}}{\end{lemma1}}

\newenvironment{conjecture}{\begin {lemma1}{\bf Conjecture.}}{\end{lemma1}}

\newcommand{\Q}{\ensuremath{\mathbb{Q}}}
\newcommand{\Z}{\ensuremath{\mathbb{Z}}}
\newcommand{\C}{\ensuremath{\mathbb{C}}}
\newcommand{\N}{\ensuremath{\mathbb{N}}}
\newcommand{\PP}{\ensuremath{\mathbb{P}}}

\newcommand{\holom}[3]{\ensuremath{#1:#2  \rightarrow #3}}

\makeatletter
\ifnum\@ptsize=0  \fi \ifnum\@ptsize=2  \fi 

\newcommand\sO{{\mathcal O}}

\title{Effective non-vanishing conjectures for projective threefolds} 
\date{November 19, 2008}

\author{Ama\"el Broustet}
\author{Andreas H\"oring}

\address{Ama\"el Broustet, Universit\'e Louis Pasteur, IRMA, 7 rue Ren\'e Descartes, 67084 Strasbourg, France}
\email{broustet@math.u-strasbg.fr}

\address{Andreas H\"oring, Universit\'e Paris 6, Institut de Math\'ematiques de Jussieu, Equipe de Topologie et G\'eom\'etrie Alg\'ebrique, 175, rue du Chevaleret, 75013 Paris, France}
\email{hoering@math.jussieu.fr}

\begin{document}

\begin{abstract}

Let $X$ be a smooth projective threefold, and let  $A$ be an ample line bundle
such that $K_X+A$ is nef.
We show that if $K_X$ or $-K_X$ is pseudo-effective, the adjoint bundle $K_X+A$ has global sections.
We also give a very short proof of the Beltrametti-Sommese conjecture in dimension three, recently proven by Fukuma:
if $A$ is  an ample line bundle
such that $K_X+2A$ is nef, the adjoint bundle $K_X+2A$ has global sections.
\end{abstract}

\maketitle

\vspace{-1ex}

\section{Introduction}

\medskip

Let $X$ be a projective complex manifold, and let $A$ be an ample line bundle such that $K_X+A$ is nef.
By the base-point free theorem the adjoint line bundle $K_X+A$ is semiample, so for $m$ sufficiently large the linear system $|m (K_X+A)|$ 
is not empty. In this note we study an effective version of this statement, conjectured by 
Ionescu \cite{Cet90} (cf. also  \cite{Amb99} and \cite{Ka00} for similar conjectures in the singular case).

\begin{conjecture}  \label{theconjecture}
Let $X$ be a projective manifold, and let $A$ be an ample line bundle such that $K_X+A$ is nef.
Then  we have
$$
H^0(X, K_X + A) \neq 0.
$$
\end{conjecture}

If $X$ is a curve, Conjecture \ref{theconjecture} is an immediate application of the Riemann-Roch theorem.
In dimension two, the result follows from Riemann-Roch and 
classical results on the second Chern class of a surface \cite{Ka00}.
In dimension three the difficulty of the problem increases considerably since we lose the control over the second Chern class: 
by a difficult theorem of Miyaoka
the second Chern class of a minimal  threefold is pseudo-effective, but if $X$ is not minimal this is no longer the case.
Nevertheless we will see that a straightforward generalisation of a theorem of Miyaoka to the case of 
$\Q$-vector bundles allows us to give a very short proof of the conjecture for threefolds that are not uniruled
(originally due to Fukuma  \cite[Thm.3.3]{Fu07}).

\begin{theorem} \label{theoremnotuniruled}
Let $X$ be a smooth projective threefold such that $K_X$ is pseudo-effective (i.e. $X$ is not uniruled). 
Then Conjecture \ref{theconjecture} holds for $X$.
\end{theorem}

Based on a detailed analysis of the properties of the second Chern class, Xie has shown that
Conjecture \ref{theconjecture} is true for 'many' threefolds with $-K_X$ nef \cite[Cor.4.5]{Xie05}. By a computational trick,
we eliminate the second Chern class from the problem and show
the conjecture for the much larger class of threefolds with pseudo-effective anticanonical bundle.

\begin{theorem} \label{theoremgenericallynef}
Let $X$ be a smooth projective threefold such that $-K_X$ is pseudo-effective.
Then Conjecture \ref{theconjecture} holds for $X$.
\end{theorem}

The proof of this theorem will not really use  the pseudoeffectivity of the anticanonical bundle, but only 
the slightly weaker condition that $-K_X$ is generically nef (cf. Definition \ref{definitiongenericallynef}).
Typically one expects this condition to hold for rationally connected varieties, but
there are even rational manifolds whose
anticanonical bundle is not generically nef (cf. Example \ref{examplebadanticanonical}).
From our point of view, Conjecture \ref{theconjecture} is much more difficult
for rationally connected threefolds with such a 'bad' anticanonical divisor
and we postpone their treatment to a sequel of this paper.

A weaker version of Conjecture \ref{theconjecture} is motivated by the adjunction theory of complex manifolds and still open in general :

\begin{conjecture} (Beltrametti-Sommese \cite[Conj. 7.2.7]{BS95}) \label{BSconjecture} 
Let $X$ be a projective manifold, and let $A$ be an ample line bundle such that $K_X + (\dim X-1) A$ is nef.
Then  we have
$$
H^0(X, K_X + (\dim X-1)A) \neq 0.
$$
\end{conjecture}

We establish this conjecture in the threefold case.

\begin{theorem} \label{theorembs}
Let $X$ be a smooth projective threefold. Then Conjecture \ref{BSconjecture} holds for $X$.
\end{theorem}

Based on difficult adjunction theoretic considerations this statement was shown by Fukuma \cite{Fu06},
but we believe that our approach provides new insight into the nature of the problem: by Kodaira vanishing
the Euler characteristic $\chi(X, K_X + A)$ is non-negative, so $\chi(X, K_X+2A) - 2 \chi(X, K_X + A)$ is a lower bound
for $h^0(X, K_X + 2A)$. Yet this lower bound can be easily computed since
the second Chern class of $X$ does not appear in the formula.

\smallskip

Note furthermore that our statements are actually much more precise than what is stated above, since we give easily computable 
lower bounds on the dimensions of the linear systems in terms of intersection numbers and the holomorphic
Euler characteristic of $X$.

\section{Notation and basic remarks}

We work over the complex field $\C$. 
For standard definitions in complex algebraic geometry we refer to  \cite{Ha77}. 
Manifolds and varieties are always supposed to be irreducible.
If $X$ is a projective manifold, we will identify line bundles, Cartier divisors and invertible sheaves.
We will denote by $N^1(X)_\Q$ the $\Q$-vector space of $\Q$-divisors modulo numerical equivalence \cite[1.3]{Deb01}.

Note that if $A$ is an ample line bundle on $X$ and $m \in \N$, then 
by the Kodaira vanishing theorem the higher cohomology of $K_X+mA$ vanishes. 
In particular we have 
$$
h^0(X, K_X + mA) = \chi(X, K_X + mA)
$$
and in our proofs we will always compute $\chi(X, K_X + mA)$
via the Grothendieck-Riemann-Roch theorem \cite[App. A]{Ha77}:
let $X$ be a smooth projective threefold and $D$ a line bundle on $X$. Then we have
\begin{equation} \label{RR3}
\chi(X, D) = \frac{1}{12} D \cdot (D-K_X) \cdot (2D-K_X) +  \frac{1}{12} D \cdot c_2(X) + \chi(X, \sO_X)
\end{equation}
and
\begin{equation} \label{chiox}
\chi(X, {\cal O}_X) = \frac{-1}{24} K_X \cdot c_2(X).
\end{equation}

We will use the formalism of Chern classes of $\Q$-vector bundles \cite[Ch.6.2, Ch.8.1]{Laz04}: 
let $X$ be a projective manifold, let $E$ be a vector bundle of rank $r$ over $X$, and let
$\delta \in N^1(X)_\Q$ be the numerical class of a $\Q$-divisor, then
\begin{equation} \label{qc1}
c_1(E\hspace{-0.8ex}<\hspace{-0.8ex}\delta\hspace{-0.8ex}>) = c_1(E) + r \delta
\end{equation}
and
\begin{equation} \label{qc2}
c_2(E\hspace{-0.8ex}<\hspace{-0.8ex}\delta\hspace{-0.8ex}>) = c_2(E) + (r-1) c_1(E) \cdot \delta + \frac{r (r-1)}{2} \delta^2.
\end{equation}

\begin{definition} \label{definitiongenericallynef}
Let $X$ be a projective manifold, and let $L$ be a $\Q$-divisor on $X$. 
$L$ is pseudo-effective if 
$L \cdot C \geq 0$ for all irreducible curves $C$ which move in a family covering $X$. 

A $\Q$-vector bundle $E$ over $X$
is generically nef if the following holds:
given ample divisors $H_1, \ldots, H_{\dim X-1}$ , let
$C$ be a curve cut out by general elements in
$|m_j H_j|$ for $m_j \in \N$ sufficiently high. Then 
$E|_C$ is nef.
\end{definition}

Note that a $\Q$-divisor $L$ is generically nef if and only if for any collection of ample divisors $H_1, \ldots, H_{\dim X-1}$, we have
\[
L \cdot H_1 \cdot \ldots \cdot H_{\dim X-1} \geq 0.
\]
In particular a pseudo-effective divisor is generically nef.

Recall that by \cite{BDPP04}, a projective manifold $X$ is not uniruled
if and only if $K_X$ is pseudo-effective.

Finally we will use some basic facts from the adjunction theory of complex threefolds \cite[Ch.7]{BS95}.

\begin{theorem} \label{theoremadjunctiontheory}
Let $X$ be a smooth projective threefold that is not uniruled, and let $A$ be an ample line bundle on $X$.  
\begin{enumerate}
\item Then the adjoint bundle $K_X+2A$ is nef. 
\item If the adjoint bundle $K_X+A$ is not nef, there exists a birational morphism
\holom{f}{X}{X'} onto a smooth projective threefold $X'$ such that $X$ is the blow-up of a point in $X'$. 
Furthermore $K_X+2A$ is $f$-trivial, so $K_X+2A=f^* H$ for some line bundle $H$. 
In particular we have $A = f^* A' - E$ where $E$ is the exceptional divisor of $f$ and $A'$ an divisor on $X'$.
\end{enumerate}
\end{theorem}

\section{The non-uniruled case}

We extend a theorem of Miyaoka \cite[Thm. 6.1]{Miy87} to the case of $\Q$-vector bundles. 

\begin{theorem} \label{theoremmiyaokaQ}
Let $X$ be a projective manifold of dimension $n$,  let $E$ be a vector bundle over $X$,
and let $\delta$ be the numerical class of a $\Q$-divisor.
If $E\hspace{-0.8ex}<\hspace{-0.8ex}\delta\hspace{-0.8ex}>$ is generically nef and $c_1( E\hspace{-0.8ex}<\hspace{-0.8ex}\delta\hspace{-0.8ex}>)$ is nef, then
$$
c_2(E\hspace{-0.8ex}<\hspace{-0.8ex}\delta\hspace{-0.8ex}>) \cdot H_1 \cdot \ldots \cdot H_{n-2} \geq 0.
$$
\end{theorem}

\begin{proof}
Fix $m \in \N$ sufficiently divisible such that $m \delta$ is the class of a line bundle $L$. 
By \cite[Thm.4.1.10]{Laz04} there exists a finite covering $f : Y \rightarrow X$ by a projective manifold $Y$ 
such that $f^* L \simeq N^{\otimes m}$, where $N$ is a line bundle on $Y$. 
In particular $f^* \delta$ is the class of the line bundle $N$.
Since  $E$ is generically nef, the 
vector bundle $f^*E$ is generically nef.
Furthermore by \cite[Lemma 8.12]{Laz04}
\[
c_1(f^* (E\hspace{-0.8ex}<\hspace{-0.8ex}\delta\hspace{-0.8ex}>)) = f^* c_1(E\hspace{-0.8ex}<\hspace{-0.8ex}\delta\hspace{-0.8ex}>), 
\]
so $c_1(f^* (E\hspace{-0.8ex}<\hspace{-0.8ex}\delta\hspace{-0.8ex}>))$ is a nef $\Q$-divisor.
Since $f^* (E\hspace{-0.8ex}<\hspace{-0.8ex}\delta\hspace{-0.8ex}>) \simeq f^*E\hspace{-0.8ex}<\hspace{-0.8ex}f^*\delta\hspace{-0.8ex}>$ 
has the same Chern classes as the vector bundle $f^*E \otimes N$, we can apply
\cite[Thm. 6.1]{Miy87} to $f^*E \otimes N$ and get
\[
f^*H_1 \cdot \ldots \cdot f^*H_{n-2} \cdot  f^* c_2(E\hspace{-0.8ex}<\hspace{-0.8ex}\delta\hspace{-0.8ex}>) = f^*H_1 \cdot \ldots \cdot f^*H_{n-2} \cdot c_2(f^*E \otimes N) \geq 0.
\]
Now $f$ is finite, so the projection formula \cite[Thm.3.2]{Ful98} implies
$$
\deg f \cdot H_1 \cdot \ldots \cdot H_{n-2} \cdot \ c_2(E\hspace{-0.8ex}<\hspace{-0.8ex}\delta\hspace{-0.8ex}>) = 
 f^*H_1 \cdot \ldots \cdot f^* H_{n-2}) \cdot f_* (f^* c_2(E\hspace{-0.8ex}<\hspace{-0.8ex}\delta\hspace{-0.8ex}>) \geq 0.
$$
\end{proof}

\begin{corollary} \label{corollarynotuniruled}
Let $X$ be a projective manifold of dimension $n$ such that $K_X$ is pseudo-effective (i.e. $X$ is not uniruled). Let $A$ be a nef $\Q$-divisor on $X$
such that $K_X+A$ is nef. Let $H_1, \ldots, H_{\dim X-2}$ be a collection of ample line bundles on $X$. Then we have
\[
H_1 \cdot  \ldots \cdot H_{n-2} \cdot c_2(X) \geq - H_1 \cdot  \ldots \cdot H_{n-2} \cdot (\frac{n-1}{n} K_X \cdot A + \frac{n-1}{2n} A^2).
\]
\end{corollary}

\begin{proof}
Since $X$ is not uniruled, the cotangent bundle $\Omega_X$ is generically nef by Miyaoka's theorem \cite[Cor.6.4]{Miy87}. 
In particular the $\Q$-vector bundle $\Omega_X\hspace{-0.8ex}<\hspace{-0.8ex}\frac{1}{n}A\hspace{-0.8ex}>$ is generically nef, 
furthermore $c_1(\Omega_X\hspace{-0.8ex}<\hspace{-0.8ex}\frac{1}{n}A\hspace{-0.8ex}>) = K_X + A$ is nef.
Therefore Theorem \ref{theoremmiyaokaQ} yields
$$
H_1 \cdot  \ldots \cdot H_{n-2} \cdot c_2(\Omega_X\hspace{-0.8ex}<\hspace{-0.8ex}\frac{1}{n}A\hspace{-0.8ex}>) \geq 0.
$$ 
Since by Formula (\ref{qc2})
$$
c_2(\Omega_X\hspace{-0.8ex}<\hspace{-0.8ex}\frac{1}{n}A\hspace{-0.8ex}>) = c_2(X) + \frac{n-1}{n} K_X \cdot A + \frac{n-1}{2n} A^2,
$$
we get
$$
H_1 \cdot  \ldots \cdot H_{n-2} \cdot c_2(X) \geq - H_1 \cdot \ldots \cdot H_{n-2} \cdot (\frac{n-1}{n} K_X \cdot A + \frac{n-1}{2n} A^2).
$$
\end{proof}

\begin{remark}
The preceding statements are essentially equivalent to \cite[Thm.2.1, Cor. 2.2.1]{Fuk05}. The use of $\Q$-vector bundles allows us to bypass
the delicate and somewhat lengthy computations used to prove these statements.
\end{remark}

We will now use Corollary \ref{corollarynotuniruled} to give effective estimates on the dimension of the linear systems
$|K_X+A|$ and $|K_X+2A|$. 

\begin{proposition} \label{propositionnotuniruled}
Let $X$ be a smooth projective threefold such that $K_X$ is pseudo-effective (i.e. $X$ is not uniruled). Let $A$ be an ample line bundle on $X$. Then
we have
\[
h^0(X, K_X+A) \geq  \frac{1}{18} (K_X+2A) \cdot  A \cdot (K_X + \frac{5}{4} A) > 0
\]
and
\[
h^0(X, K_X+2A) - h^0(X, K_X+A) \geq \frac{1}{12} [A \cdot (K_X+2A) \cdot (K_X+ \frac{19}{3} A) + A^3] > 0.
\] 
\end{proposition}

\begin{remark} The first part of the proposition implies Theorem \ref{theoremnotuniruled}: in fact it shows
that if $X$ is not uniruled, we do not need the assumption that $K_X+A$ is nef.

Note furthermore that if $K_X$ is pseudo-effective and $K_X+A$ is nef, then
\[
 \frac{1}{18} (K_X+2A) \cdot  A \cdot (K_X + \frac{5}{4} A) = \frac{5}{36} A^3+ \frac{1}{8} K_X \cdot A^2 + \frac{1}{18} K_X \cdot (K_X+A) \cdot A
\geq  \frac{5}{36} A^3+ \frac{1}{8} K_X \cdot A^2
\]
so we slightly refine the lower bound given in \cite[Thm.3.2]{Fu07}.
For a thorough discussion of the optimality of the bound, cf. the examples in \cite[Ex.3.1]{Fu07}.
\end{remark}

\begin{proof}
{\it Step 1: Reduction to the case where $K_X+A$ is nef.}
If $K_X+A$ is not nef, let \holom{f}{X}{X'} be the birational morphism given by Theorem \ref{theoremadjunctiontheory}. 
We have $A = f^* A' - E$ where $E$ is the exceptional divisor of $f$ and $A'$ a line bundle on $X'$.
The Nakai-Moishezon criterion immediately shows 
that $A'$ is ample, moreover $A'^3=A^3+1$.
Since $K_X+A=f^*(K_{X'}+A')+E$ and  $K_X+2A= f^* (K_{X'}+2A')$, we get
$$
H^0(X, K_X+A) \simeq H^0(X', K_{X'}+A'), \ \ H^0(X, K_X+2A) \simeq H^0(X', K_{X'}+2A').
$$
Note furthermore that since $K_X+2A= f^* (K_{X'}+2A')$, we have
\[
(K_X+2A) \cdot A \cdot (K_X+\frac{5}{4} A)=
(K_{X'}+2A') \cdot A' \cdot (K_{X'}+\frac{5}{4} A')
\]
and
\[
A \cdot (K_X+2A) \cdot (K_X+\frac{19}{3} A)=
A' \cdot (K_{X'}+2A') \cdot (K_{X'}+\frac{19}{3} A')
\]
so it is sufficient to establish the statement on $X'$. Since the Picard number of $X'$ is strictly smaller than the Picard number of $X$, we conclude by induction.

{\it Step 2: proof of the first statement.} $K_X+A$ is nef by Step 1, so in particular the line bundle $K_X+2A$ is ample.
Therefore by Corollary \ref{corollarynotuniruled}
\[
(*) \qquad (K_X+2A) \cdot c_2(X) \geq -(K_X+2A) \cdot (\frac{2}{3} K_X \cdot A + \frac{1}{3} A^2).
\]
By Riemann-Roch (\ref{RR3}) and Formula (\ref{chiox}) we have
\[
\chi(X, K_X+A) 
= 
\frac{1}{12} (K_X+A) \cdot A \cdot (K_X+2A) +  \frac{1}{24} (K_X+2A) \cdot c_2(X). 
\]
So by $(*)$ the right hand side is bounded below by 
\[
\frac{1}{12} (K_X+A) \cdot A \cdot (K_X+2A) - \frac{1}{24} (K_X+2A) \cdot (\frac{2}{3} K_X \cdot A + \frac{1}{3} A^2)
= \frac{1}{18} (K_X+2A) \cdot  A \cdot (K_X + \frac{5}{4} A).
\]
Since $K_X+A$ is nef, the $\Q$-divisor $K_X+\frac{5}{4} A$ is ample, so this expression is strictly positive.

{\it Step 3: proof of the second statement.} Since $K_X+A$ is nef by Step 1, 
we have by Corollary \ref{corollarynotuniruled} 
\[
(**) \qquad A \cdot c_2(X) \geq -A \cdot (\frac{2}{3} K_X \cdot A + \frac{1}{3} A^2).
\]
By Riemann-Roch (\ref{RR3}) we have
$$
\chi(X, K_X+A) = \frac{1}{12} (K_X+A) \cdot A \cdot (K_X+2A) +  \frac{1}{12} (K_X+A) \cdot c_2(X) + \chi(X, \sO_X)
$$
and
$$
\chi(X, K_X+2A) = \frac{1}{12} (K_X+2A) \cdot 2A \cdot (K_X+4A) +  \frac{1}{12} (K_X+2A) \cdot c_2(X) + \chi(X, \sO_X).
$$
Therefore
$$
\chi(X, K_X+2A) - \chi(X, K_X+A)
= \frac{1}{12} (K_X+2A) \cdot A \cdot (K_X+7A) +  \frac{1}{12} A \cdot c_2(X)
$$
and by $(**)$ the right hand side is bounded below by 
\[
 \frac{1}{12} [(K_X+2A) \cdot A \cdot (K_X+7A) - A \cdot (\frac{2}{3} K_X \cdot A + \frac{1}{3} A^2)]
=
 \frac{1}{12} [A \cdot (K_X+2A) \cdot (K_X+ \frac{19}{3} A) + A^3].
\]
Since $K_X+A$ is nef, the line bundle $K_X+2A$ is ample, so this expression is strictly positive.
\end{proof}

\section{Proof of Theorem \ref{theoremgenericallynef}}

\begin{lemma} \label{lemmaestimate}
Let $X$ be a smooth projective threefold, 
and let $A$ be an ample line bundle such that $K_X+A$ is nef and big. Then we have
$$
h^0(X, K_X+A) \geq - \frac{1}{2} K_X \cdot (K_X+A)^2  + 2 \chi(X, \sO_X).
$$
\end{lemma}

\begin{remark} The bound given in the lemma is sharp: set $X=\PP^3$ and $A=\sO_{\PP^3}(5)$, then 
$$
h^0(\PP^3, K_{\PP^3}+A)= h^0(\PP^3, \sO_{\PP^3}(1)) = 4
$$
and  
$$
- \frac{1}{2} K_{\PP^3} \cdot (K_{\PP^3}+A)^2  + 2 \chi(\PP^3, \sO_{\PP^3}) = 4.
$$
\end{remark}

\begin{proof} 
Since $K_X+A$ is nef and big, the higher cohomology of $2 K_X+A$ vanishes by the Kawamata-Viehweg vanishing theorem, thus
$\chi(X, 2 K_X+A)\geq 0$ and
$$
\chi(X, K_X+A) \geq \chi(X, K_X+A) - \chi(X, 2 K_X+A).
$$
By Riemann-Roch (\ref{RR3}) we have
$$
\chi(X, K_X+A) = \frac{1}{12} (K_X+A) \cdot A \cdot (K_X+2A) +  \frac{1}{12} (K_X+A) \cdot c_2(X) + \chi(X, \sO_X).
$$
Again by Riemann-Roch  (\ref{RR3})  and Formula (\ref{chiox}), we get
\[
\chi(X, 2 K_X+A) 
= 
\frac{1}{12} (2 K_X+A) \cdot (K_X+A) \cdot (3 K_X+2A) +  \frac{1}{12} (K_X+A) \cdot c_2(X) - \chi(X, \sO_X),
\]
thus 
$$
\chi(X, K_X+A) - \chi(X, 2 K_X+A) 
= - \frac{1}{2} K_X \cdot (K_X+A)^2 + 2 \chi(X, \sO_X).
$$
\end{proof}

\begin{proof}[Proof of Theorem \ref{theoremgenericallynef}]
By \cite[Thm.3.1]{Ka00} the statement holds if $K_X+A$ is nef but not big, so we can suppose without loss of generality 
that $K_X+A$ is nef and big. Furthermore we may suppose that $X$ is uniruled, since otherwise we conclude with Theorem \ref{theoremnotuniruled}.
In particular we have by Serre duality  $0=h^0(X, K_X)=h^3(X, \sO_X)$. If $h^1(X, \sO_X) > 0$, we conclude with  \cite[Thm.4.2]{CH02}.
Thus we can suppose 
$$
\chi(X, \sO_X) = 1 + h^2(X, \sO_X) \geq 1. 
$$
Since $K_X+A$ is nef and big, the $\Q$-divisor $K_X+A+ \varepsilon H$ is ample for every $H$ ample and $\varepsilon >0$.
Since $-K_X$ is pseudo-effective (hence generically nef), we have
$$
-K_X \cdot (K_X+A + \varepsilon H) \cdot (K_X+A + \varepsilon H) \geq 0
$$ 
So Lemma  \ref{lemmaestimate} above implies the claim.
\end{proof}

The following example shows that there are rationally connected manifolds whose anticanonical bundle is not generically nef.
Note that this implies that the tangent bundle of $X$  is not generically nef (cf. \cite{Pet08}).

\begin{example} \label{examplebadanticanonical}
Let $(F_t)_{t \in \PP^1}$ be a pencil of generic hypersurfaces of degree at least $n+2$ in $\PP^{n}$.  The general member of the pencil is smooth
and by the adjunction formula
the canonical bundle of $F_t$ is ample. Let \holom{\mu}{X}{\PP^n} be a birational morphism that resolves the base points
of the pencil so that we have a fibration \holom{f}{X}{\PP^1} such that the general fibre $F$ is a member of the pencil. We claim 
that the anticanonical bundle of $X$ is not generically nef:
let $H$ be an ample line bundle on $X$ then for every rational number $\varepsilon>0$ the $\Q$-divisor $F+\varepsilon H$ is
ample. Since
$$
K_X \cdot (F+\varepsilon H)^{n-1} = (n-1) \varepsilon^{n-2} K_X \cdot F \cdot H^{n-2} +  \varepsilon^{n-1} K_X \cdot H^{n-1} 
$$
and  $K_X \cdot F \cdot H^{n-2}= K_F \cdot H|_F^{n-2} > 0$ we can choose a small rational $\varepsilon>0$ such that
$K_X \cdot (F+\varepsilon H)^{n-1}>0$. Let $m \in \N$ be sufficiently large and divisible such that $m (F+\varepsilon H)$
is linearly equivalent to a $\Z$-divisor $A$. Then the preceding computation shows that for every $m_1, \ldots, m_{n-1} \in \N$,
the anticanonical divisor $-K_X$ is $(m_1 A, \ldots, m_{n-1} A)$-antiample, that is
\[
-K_X \cdot (m_1 A) \cdot \ldots \cdot (m_{n-1} A) < 0.
\]
\end{example}

\section{Beltrametti-Sommese conjecture}

\begin{lemma} \label{lemmabschi}
Let $X$ be a smooth projective threefold, and let $A$ be an ample line bundle on $X$. Then
$$
h^0(X, K_X+2A) \geq  \frac{1}{2} (K_X+2A) \cdot A^2 + \chi(X, \sO_X).
$$ 
\end{lemma}

\begin{remark} The bound given in the lemma is sharp: set $X=\PP^3$ and $A=\sO_{\PP^3}(3)$, then 
$$
h^0(\PP^3, K_{\PP^3}+2A) =  h^0(\PP^3, \sO_{\PP^3}(2)) = 10
$$
and  
$$
\frac{1}{2} (K_{\PP^3}+2A) \cdot A^2 + \chi(\PP^3, \sO_{\PP^3}) = 10.
$$
\end{remark}

\begin{proof}
By Kodaira vanishing  the higher cohomology of $K_X+A$ vanishes, thus
$\chi(X, K_X + A) \geq 0$ and hence
\[
\chi(X, K_X+2A) \geq \chi(X, K_X+2A) - 2 \chi(X, K_X + A).
\]
By Riemann-Roch (\ref{RR3}) and Formula (\ref{chiox}) we have
$$
\chi(X, K_X+A) = \frac{1}{12} (K_X+A) \cdot A \cdot (K_X+2A) +  \frac{1}{12} A \cdot c_2(X) - \chi(X, \sO_X)
$$
and
$$
\chi(X, K_X+2A) = \frac{1}{12} (K_X+2A) \cdot 2A \cdot (K_X+4A) +  \frac{1}{6} A \cdot c_2(X) - \chi(X, \sO_X).
$$
Therefore
\[
\chi(X, K_X+2A) - 2 \chi(X, K_X + A) =  \frac{1}{2} (K_X+2A) \cdot A^2 + \chi(X, \sO_X).
\]
\end{proof}

\begin{proof}[Proof of Theorem \ref{theorembs}]
If $K_X+2A$ is numerically trivial, the anticanonical divisor $-K_X \equiv_{num} 2A$ is ample, so $X$ is a Fano manifold. 
In particular the numerically trivial bundle $K_X+2A$ is trivial and thus has a section.

Suppose now that $K_X+2A$ is nef and not numerically trivial.
By the base-point free theorem a positive multiple of $K_X+2A$ has a global section, 
so this implies
$$
(K_X+2A) \cdot A^2 > 0.
$$
Furthermore we may suppose that $X$ is uniruled, since otherwise we conclude with Theorem \ref{theoremnotuniruled}.
In particular we have by Serre duality  $0=h^0(X, K_X)=h^3(X, \sO_X)$. If $h^1(X, \sO_X) > 0$, we conclude with  \cite[Thm.4.2]{CH02}.
Thus we can suppose 
$$
\chi(X, \sO_X) = 1 + h^2(X, \sO_X) \geq 1. 
$$
Conclude with Lemma \ref{lemmabschi}.
\end{proof}

\end{document}